\documentclass[review]{elsarticle}
\usepackage{lineno,hyperref}
\usepackage{amsmath,amssymb}
\usepackage{amsfonts}
\usepackage{verbatim}
\usepackage{url}
\usepackage{graphicx}
\usepackage{float}
\usepackage{mathrsfs}
\usepackage{epstopdf}
\usepackage{cancel}
\usepackage{MnSymbol}
\usepackage{tikz}
\usepackage[margin=1in]{geometry}

\usetikzlibrary{arrows}
\modulolinenumbers[5]
\allowdisplaybreaks
\newtheorem{theorem}{Theorem}

\newtheorem{example}[theorem]{Example}
\newtheorem{lemma}[theorem]{Lemma}

\newtheorem{remark}[theorem]{Remark}
\newtheorem{problem}[theorem]{Problem}

\newproof{proof}{Proof}
\numberwithin{equation}{subsection}
\numberwithin{theorem}{subsection}
\newcommand{\perm}{\mathrm{perm}}
\begin{document}

\begin{frontmatter}

\title{A Characterization of Oriented Hypergraphic Laplacian and Adjacency Matrix Coefficients}

\author[add1]{Gina Chen\fnref{fn1}}
\author[add1]{Vivian Liu\fnref{fn1}}
\author[add2]{Ellen Robinson\fnref{fn2}}
\author[add2]{Lucas J. Rusnak\corref{mycorrespondingauthor}}
\ead{Lucas.Rusnak@txstate.edu}
\author[add1]{Kyle Wang\fnref{fn1}}

\address[add2]{Department of Mathematics, Texas State University, San Marcos, TX 78666, USA}

\address[add1]{Mathworks, Texas State University, San Marcos, TX 78666, USA}

\fntext[fn1]{Portions of these results submitted to the 2016 Siemens Competition (regional semi-finalist).}
\fntext[fn2]{Portions of these results appear in 2017 Master's Thesis.}

\cortext[mycorrespondingauthor]{Corresponding author}

\begin{abstract}

An oriented hypergraph is an oriented incidence structure that generalizes and unifies graph 
and hypergraph theoretic results by examining its locally signed graphic substructure. In this 
paper we obtain a combinatorial characterization of the coefficients of the characteristic polynomials 
of oriented hypergraphic Laplacian and adjacency matrices via a signed hypergraphic generalization of 
basic figures of graphs. Additionally, we provide bounds on the determinant and permanent of the 
Laplacian matrix, characterize the oriented hypergraphs in which the upper 
bound is sharp, and demonstrate that the lower bound is never achieved.

\end{abstract}


\end{frontmatter}


\section{Introduction}

Sachs' Coefficient Theorem provides a combinatorial interpretation of the
coefficients of the characteristic polynomial of the adjacency matrix of a
graph as families of sub-graphs \cite{SGBook}, this was recently extended to
signed graphs in \cite{Sim1}. In this paper we obtain an oriented
hypergraphic generalization of Sachs' Coefficient Theorem that extends to
the oriented hypergraphic Laplacian and the signless Laplacian. This
extension shows that the standard adjacency matrix coefficients are the
restricted enumeration of a family of sub-incidence-structures associated to
any finite integral incidence matrix --- providing a single class of
combinatorial objects to study the coefficients of both characteristic
polynomials.

These theorems are unified and generalized by using the weak walk Theorem
for oriented hypergraphs in \cite{AH1, OHHar}, which unifies the entries of
the oriented hypergraphic matrices as weak walk counts, then constructing
incidence preserving maps from disjoint $1$-paths into a given oriented
hypergraph. Restrictions of these maps to adjacency preserving maps on
sub-oriented-hypergraphs obtained by weak deletion of vertices allows for
the reclaiming of basic figures of graphs, as well as the determinant of the
adjacency matrix by cycle covers.

The necessary oriented hypergraphic background\ and Sachs' Theorem are
collected in Section \ref{background}. In Section \ref{IPM} we examine the
relationship between incidence preserving maps, weak walks, and generalized
cycle covers called \emph{contributors}. Section \ref{main} establishes the
permanent and determinant of the adjacency and Laplacian matrices as
contributor counts as well as the main coefficient theorems of determinant
and permanent versions of the characteristic polynomials. Finally,
contributor counts are used to provide upper and lower bounds for the
determinant and permanent of the Laplacian matrix over all orientations of a
given underlying hypergraph. The lower bound is shown to never be sharp,
while the family of oriented hypergraphs that achieve the upper bound are
characterized.

\section{Background\label{background}}

\subsection{Oriented Hypergraphs}

These definitions are an adaptation of those appearing in \cite{OH1, AH1,
OHHar}, and allow for oriented hypergraphs to be treated as locally signed
graphic through its adjacencies.

An \emph{oriented hypergraph} is a quintuple $(V,E,\mathcal{I},\iota ,\sigma
)$ where $V$, $E$, and $\mathcal{I}$ denote disjoint sets of \emph{vertices}%
, \emph{edges}, and \emph{incidences}, respectively, with incidence function 
$\iota :\mathcal{I}\rightarrow V\times E$, and orientation function $\sigma :%
\mathcal{I}\rightarrow \{+1,-1\}$. We say $v$\emph{\ and }$e$\emph{\ are
incident along }$i$ if $\iota (i)=(v,e)$. Two incidences $i$ and $j$ are
said to be \emph{parallel} if $\iota (i)=\iota (j)$ --- this provides an
equivalence class of parallel incidences where the size of each equivalence
class is called the \emph{multiplicity of the incidence}.

\begin{figure}[H]
\centering
\includegraphics{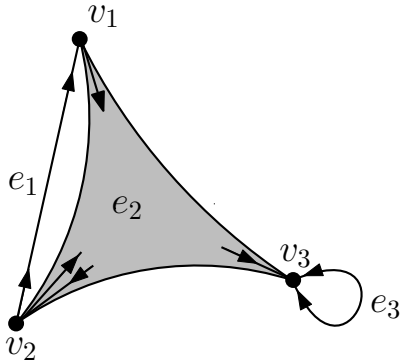}
\caption{An oriented hypergraph.}
\label{ExOH}
\end{figure}

The \emph{degree of vertex }$v$ is $\deg (v):=\left\vert \{i\in \mathcal{I}%
\mid (proj_{V}\circ \iota )(i)=v\}\right\vert $, while the \emph{size of an
edge }$e$ is $size(e):=\left\vert \{i\in \mathcal{I}\mid (proj_{E}\circ
\iota )(i)=e\}\right\vert $. Vertices $v$ and $w$ are said to be \emph{%
adjacent with respect to edge }$e$ if there are incidences $i\neq j$ such
that $\iota (i)=(v,e)$ and $\iota (j)=(w,e)$. A \emph{directed adjacency} is
a quintuple $(v,i,e,j,w)$ where $v$ and $w$ are adjacent with respect to
edge $e$ using incidences $i$ and $j$. Observe that if the directed
adjacency $(v,i,e,j,w)$ exists, then the opposite directed adjacency $%
(w,j,e,i,v)$ also exists. An \emph{adjacency} is the set associated to a
directed adjacency (or its opposite). Dropping the symmetric directedness
condition allows for directed oriented hypergraphic results to be studied.
The \emph{sign of the adjacency }$(v,i,e,j,w)$ is%
\begin{equation*}
sgn(v,i,e,j,w)=-\sigma (i)\sigma (j)\text{,}
\end{equation*}%
and $sgn(v,i,e,j,w)=0$ if $v$ and $w$ are not adjacent. \emph{Weak deletion
of vertex }$v$ is the oriented hypergraph obtained by deleting vertex $v$
and all incidences containing $v$ --- note this does not delete edges.

\subsection{Weak Walks}

A \emph{(directed) weak walk} is a sequence 
\begin{equation*}
W=(a_{0},i_{1},a_{1},i_{2},a_{2},i_{3},a_{3},...,a_{n-1},i_{n},a_{n})
\end{equation*}%
of vertices, edges and incidences, where $\{a_{k}\}$ is an alternating
sequence of vertices and edges, and $i_{k}$ is an incidence between $a_{k-1}$
and $a_{k}$; specifically, $\{(proj_{V}\circ \iota )(i_{k}),(proj_{E}\circ
\iota )(i_{k})\}=\{a_{k-1},a_{k}\}$. Similar to directed adjacencies, a 
\emph{weak walk} is the set associated to a directed weak walk. The prefix 
\emph{vertex/edge/cross} is used when the end points of a weak walk are
vertices/edges/one edge and one vertex. The \emph{length of a weak walk} is
half the number of incidences in the weak walk.

A \emph{vertex walk} is a weak walk where $a_{0}$, $a_{n}\in V$, and $%
i_{2k-1}\neq i_{2k}$, and an adjacency is a vertex walk of length $1$. A 
\emph{vertex backstep} is a weak walk of length $1$ of the form $(v,i,e,i,v)$%
, while a \emph{loop} is a vertex walk of the form $(v,i,e,j,v)$ where $%
i\neq j$. A \emph{vertex path} is a vertex walk where no vertex or edge is
repeated, while a \emph{circle} is a vertex path except $a_{0}=a_{n}$.
Analogous edge-centric definitions exist for the incidence dual and the
results are inherited.

The \emph{sign of a weak walk} $W$ is%
\begin{equation*}
sgn(W)=(-1)^{\lfloor n/2\rfloor }\prod_{h=1}^{n}\sigma (i_{h})\text{,}
\end{equation*}%
which is equivalent to taking the product of the signed adjacencies if $W$
is a vertex walk; see \cite{SG, OSG, MR0267898} for bidirected graphs as
orientations of signed graphs.

\subsection{Oriented Hypergraphic Matrices}

\label{OHM}

The vertices and edges of an oriented hypergraph $G$ are regarded as totally
ordered as the row and column labels of a given $V\times E$ \emph{incidence
matrix} $\mathbf{H}_{G}$ where the $(v,e)$-entry is the sum of all $\sigma
(i)$ such that $\iota (i)=(v,e)$. The \emph{adjacency matrix} $\mathbf{A}%
_{G} $ of an oriented hypergraph $G$ is the $V\times V$ matrix whose $(v,w)$%
-entry is the sum of all signed adjacencies from $v$ to $w$. The \emph{%
degree matrix} of an oriented hypergraph $G$ is the $V\times V$ diagonal
matrix $\mathbf{D}_{G}:=diag(\deg (v_{1}),\ldots ,\deg (v_{n}))$. The \emph{%
Laplacian matrix of }$G$ is defined as $\mathbf{L}_{G}:=\mathbf{H}_{G}%
\mathbf{H}_{G}^{T}=\mathbf{D}_{G}-\mathbf{A}_{G}$ for all oriented
hypergraphs.

The \emph{(vertex) weak walk matrix of length }$k$ is the matrix $\mathbf{W}%
_{(G,V,V,k)}$ whose $(v,w)$-entry is the number of positive weak walks of
length $k$ from $v$ to $w$ minus the number of negative weak walks of length 
$k$ from $v$ to $w$. It was shown in \cite{AH1} for incidence-simple
oriented hypergraphs, and improved to all oriented hypergraphs in \cite%
{OHHar}, that the entries of oriented hypergraphic matrices are weak walk
counts.

\begin{theorem}
\label{WWL}Let $G$ be an oriented hypergraph.

\begin{enumerate}
\item The $(v,w)$-entry of $\mathbf{D}_{G}$ is the number of strictly weak,
weak walks, of length $1$ from $v$ to $w$. That is, the number of backsteps
from $v$ to $w$.

\item The $(v,w)$-entry of $\mathbf{A}_{G}$ is the number of positive
(non-weak) walks of length $1$ from $v$ to $w$ minus the number of negative
(non-weak) walks of length $1$ from $v$ to $w$.

\item The $(v,w)$-entry of $-\mathbf{L}_{G}$ is the number of positive weak
walks of length $1$ from $v$ to $w$ minus the number of negative weak walks
of length $1$ from $v$ to $w$. That is, $-\mathbf{L}_{G}=\mathbf{W}%
_{(G,V,V,1)}$.
\end{enumerate}
\end{theorem}

Powers of these matrices extend to weak walks of length $k$, for
non-negative integers $k$, while entries of $\mathbf{H}_{G}$ are half-walks.

\subsection{Sachs' Theorem}

Sachs' Theorem provides a combinatorial count for the coefficients of the
characteristic polynomial of the adjacency matrix of a graph, for an
expanded development see \cite{SGBook}.

An \emph{elementary figure} is either a link graph or a cycle on $n$
vertices where $n\geq 1$. A \emph{basic figure} $U$ is a graph that is the
disjoint union of elementary figures. Let $\mathscr{U}_{k}$ denote the set
of all basic figures that are contained in $G$ and have exactly $k$ isolated
vertices, let $p(U)$ be the number of elementary figures of $U$ and let $%
c(U) $ denote the number of circuits in $U$.

\begin{theorem}[Sachs' Theorem]
\label{Sachs}For a graph $G$ with $n=|V(G)|$,%
\begin{equation*}
\chi _{G}(A,x)=\sum\limits_{k=1}^{n}\left( \sum\limits_{U\in \mathscr{U}%
_{k}}(-1)^{p(U)}(2)^{c(U)}\right) x^{k}\text{.}
\end{equation*}
\end{theorem}

Sachs' Theorem has been generalized to signed graphs in \cite{Sim1}. While
alternate proofs of the Laplacian and signless Laplacian of a graph also
appear in \cite{SGBook} and \cite{Sim1}, we generalize Sachs' Theorem to
oriented hypergraphs for both the adjacency and Laplacian matrices (the
signless Laplacian is contained as a specific orientation) using a single
universal combinatorial interpretation of coefficients for both determinant
and permanent based characteristic polynomial.

\section{Incidence Preserving Maps\label{IPM}}

\subsection{Contributors as Generalized Basic Figures}

The incidence preserving maps introduced in this section are motivated by
the directed graphic version of monic, adjacency-preserving, path embeddings
into circuit graphs appearing in \cite{Grill1} as a study of $\Phi $%
-injectivity classes.

Given hypergraphs $H=(V_{H},E_{H},\mathcal{I}_{H},\iota _{H})$ and $%
G=(V_{G},E_{G},\mathcal{I}_{G},\iota _{G})$ with no isolated vertices or $0$%
-edges, an \emph{incidence preserving map} is a function $\alpha
:H\rightarrow G$ such that the following diagram commutes:

\begin{equation*}
\begin{tikzpicture}[node distance = 1.25in, auto] \node (I1)
{$\mathcal{I}_H$}; \node (I2) [right of=I1] {$\mathcal{I}_G$}; \node (VE1)
[below of=I1] {$V_H \times E_H$}; \node (VE2) [below of=I2] {$V_G \times
E_G$}; \draw[->] (I1) to node {$\alpha_\mathcal{I}$} (I2); \draw[->] (I1) to
node [swap] {$\iota_H$} (VE1); \draw[->] (VE1) to node [swap] {$(\alpha_V
\times \alpha_E)$} (VE2); \draw[->] (I2) to node {$\iota_G$} (VE2);
\end{tikzpicture}
\end{equation*}

\begin{lemma}
Let $\overrightarrow{P}_{k}$ be a directed vertex path graph of length $k$. $%
W$ is a vertex weak walk of length $k$ in $G$ if, and only if, there is an
incidence preserving map $\omega :\overrightarrow{P}_{k}\rightarrow G$ such
that $\omega (\overrightarrow{P}_{k})=W$.
\end{lemma}

\begin{proof}

Let $\overrightarrow{P}_{k}$ be the directed vertex path 
\begin{equation*}
(v_{0},i_{1},e_{1},i_{2},v_{1},i_{3},e_{2},...,e_{k},i_{2k},v_{k})
\end{equation*} 
and let
\begin{equation*}
W=(a_{0},j_{1},a_{1},j_{2},a_{2},j_{3},a_{3},...,a_{2k-1},j_{2k},a_{2k})
\end{equation*} 
be a vertex weak walk in $G$. The map $\omega :\overrightarrow{P}%
_{k}\rightarrow G$ where $\omega (v_{b})=a_{2b}$, $\omega (e_{b})=a_{2b-1}$, 
$\omega (i_{b})=j_{b}$ is the unique incidence preserving map from $%
\overrightarrow{P}_{k}$ to $W$.

Moreover, if $\overrightarrow{P}_{k}$ maps into $G$ via an incidence
preserving map $\omega $, then $\omega (\overrightarrow{P}_{k})$ is
determined by the sequence of (possibly repeating) incidences in $G$, hence
is a weak walk in $G$. \qed
\end{proof}

From here we are able to restate the weak walk Theorem for $\mathbf{L}_{G}$
from \cite{AH1, OHHar} in terms of incidence preserving maps.

\begin{theorem}
\label{WWT}The $(v,w)$-entry of $\mathbf{L}_{G}$ is $\dsum\limits_{\omega
\in \Omega _{1}}-sgn(\omega (\overrightarrow{P}_{1}))$, where $%
\overrightarrow{P}_{1}=(t,i,e,j,h)$ and $\Omega _{1}$ is the set of all
incidence preserving maps $\omega :\overrightarrow{P}_{1}\rightarrow G$ with 
$\omega (t)=v$ and $\omega (h)=w$.
\end{theorem}

A \emph{contributor of }$G$ is an incidence preserving map from a disjoint
union of $\overrightarrow{P}_{1}$'s into $G$ defined by $c:\dcoprod%
\limits_{v\in V}\overrightarrow{P}_{1}\rightarrow G$ such that $c(t_{v})=v$
and $\{c(h_{v})\mid v\in V\}=V$. Let $\mathfrak{C}(G)$ denote the set of
contributors.

By definition, each contributor creates a natural bijection from the vertex
set to itself.

\begin{lemma}
\label{CisPerm}Every contributor $c$ is associated to a single permutation $%
\pi \in S_{V}$, the symmetric group on vertex set $V$.
\end{lemma}

Two contributors that are associated to the same permutation $\pi $ are said
to be $\pi $\emph{-permutomorphic}, let $\mathfrak{C}_{\pi }(G)$ denote the
equivalence class of $\pi $-permutomorphic contributors. Observe that $\pi $%
-permutomorphic contributors need not be isomorphic for if $c(h_{v})=v$ the
associated algebraic $1$-cycle may be a result of either a loop or a
backstep. Similarly, an algebraic $2$-cycle may arise from a repeated
adjacency or two distinct adjacencies. An algebraic $2$-cycle that
corresponds to a repeated adjacency is called a \emph{degenerate 2-circle}.
This is an important distinction because a $2$-circle has two possible cycle
orientations, while a degenerate $2$-circle has only one orientation.

\begin{lemma}
\label{CPermareDiff}Permutomorphic contributors are isomorphic up to
interchanging backsteps and loops, and interchanging of $2$-circles and
degenerate $2$-circles.
\end{lemma}

Figure \ref{PermutoClass} below contains some contributors of $G$ from
Figure \ref{ExOH} grouped by permutation classes.

\begin{figure}[H]
\centering
\includegraphics{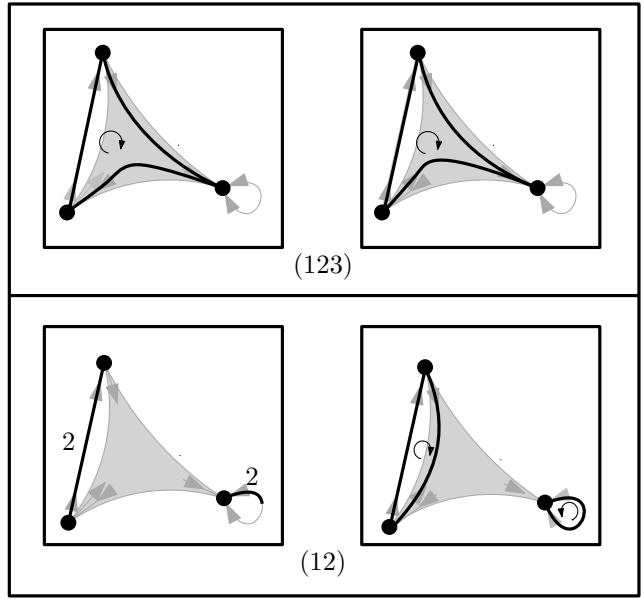}
\caption{Examples of elements in $\mathfrak{C}_{(123)}(G)$ and $\mathfrak{C}%
_{(12)}(G)$ for the oriented hypergraph in Figure \protect\ref{ExOH}.}
\label{PermutoClass}
\end{figure}

Let $\mathfrak{C}_{=k}(G)$ denote the set of contributors of $G$ with
exactly $k$ backsteps, and let $\mathfrak{C}_{\geq k}(G)$ denote the set of
contributors of $G$ with at least $k$ backsteps. Define $\widehat{\mathfrak{C%
}}_{=k}(G)$ as the collection of sub-contributors of $G$ formed from the
contributors of $\mathfrak{C}_{=k}(G)$ by deleting exactly $k$ backsteps
while retaining the $k$ isolated vertices. Similarly, define $\widehat{%
\mathfrak{C}}_{\geq k}(G)$ as the collection of sub-contributors of $G$
formed from the contributors of $\mathfrak{C}_{\geq k}(G)$ by deleting
exactly $k$ backsteps while retaining the $k$ isolated vertices. Observe
that $\mathfrak{C}_{\geq 0}(G)=\widehat{\mathfrak{C}}_{\geq 0}(G)=\mathfrak{C%
}(G)$, while $\widehat{\mathfrak{C}}_{=k}(G)$ generalizes basic figures of a
graph.

\begin{lemma}
For a graph $G$, $\widehat{\mathfrak{C}}_{=k}(G)$ is the set of oriented
basic figures with $k$ isolated vertices.
\end{lemma}

A \emph{cycle cover} of a graph $G$ is a union of disjoint cycles which are
subgraphs of $G$ and contain all of the vertices of $G$. Notice that the
cycle covers of a graph are simply the contributors that do not contain any
backsteps.

\begin{lemma}
For a graph $G$, $\widehat{\mathfrak{C}}_{=0}(G)=\mathfrak{C}_{=0}(G)$ is
the set of oriented cycle covers.
\end{lemma}

We say two elements of $\widehat{\mathfrak{C}}_{=k}(G)$ (resp. $\widehat{%
\mathfrak{C}}_{\geq k}(G)$) are $\pi $\emph{-permutomorphic} if they extend
to the same $\pi $-permutomorphism class $\mathfrak{C}_{\pi }(G)$ via the
introduction of $k$ backsteps that exist in $G$. Let $\widehat{\mathfrak{C}}%
_{=k,\pi }(G)$ (resp. $\widehat{\mathfrak{C}}_{\geq k,\pi }(G)$) be the set
of $\pi $-permutomorphic elements of $\widehat{\mathfrak{C}}_{=k}(G)$ (resp. 
$\widehat{\mathfrak{C}}_{\geq k}(G)$).

\section{General Coefficient Theorems \label{main}}

\subsection{Permanents and Determinants}

Let $ec(c)$, $oc(c)$, $pc(c)$ and $nc(c)$ be the number of even, odd,
positive, and negative circles in a contributor $c$, respectively. The
necessity for these counts are related to the differences between balanced
(all circles positive) signed graphs and the classical development of
hypergraph theory. An ordinary graph may be regarded as a signed graph with
all edges positive, hence all circles are positive, which leads to many
graph theoretic theorems having a balanced\ signed graphic analogs (see \cite%
{SG, OSG}). The classical development of hypergraphs (see \cite{Berge1,
Berge2}) uses a $\{0,1\}$-incidence matrix equivalent where a circle is
positive if, and only if, it is even. Oriented hypergraphs allow for a
locally signed graphic approach to separate the concepts of even from
positive and odd from negative; additional combinatorial properties and
applications can be found in \cite{DBM, BM, ShiBrz}.

\begin{theorem}
\label{LAPD}Let $G$ be an oriented hypergraph with adjacency matrix $\mathbf{%
A}_{G}$ and Laplacian matrix $\mathbf{L}_{G}$, then

\begin{enumerate}
\item $\mathrm{perm}(\mathbf{L}_{G})=\dsum\limits_{c\in \mathfrak{C}_{\geq
0}(G)}(-1)^{oc(c)+nc(c)}$,

\item $\det (\mathbf{L}_{G})=\dsum\limits_{c\in \mathfrak{C}_{\geq
0}(G)}(-1)^{pc(c)}$,

\item $\mathrm{perm}(\mathbf{A}_{G})=\dsum\limits_{c\in \mathfrak{C}%
_{=0}(G)}(-1)^{nc(c)}$,

\item $\det (\mathbf{A}_{G})=\dsum\limits_{c\in \mathfrak{C}%
_{=0}(G)}(-1)^{ec(c)+nc(c)}$.
\end{enumerate}
\end{theorem}

\begin{proof}

\textit{Proof of 1. }From the definition and Theorem \ref{WWT} we have%
\begin{equation*}
\mathrm{perm}(\mathbf{L}_{G})=\sum\limits_{\pi \in S_{V}}\prod\limits_{v\in
V}\dsum\limits_{\omega \in \Omega _{1,\pi }}-sgn(\omega (\overrightarrow{P}%
_{1}))\text{,}
\end{equation*}%
where $\Omega _{1,\pi }$ is the set of all incidence preserving maps $\omega
:\overrightarrow{P}_{1}\rightarrow G$ with $\omega (t)=v$ and $\omega
(h)=\pi (v)$.

Distributing the inner sums for all $v\in V$ (without evaluating the inner
sums), passes from incidence preserving maps $\omega :\overrightarrow{P}%
_{1}\rightarrow G$ with $\omega (t)=v$ and $\omega (h)=\pi (v)$ to incidence
preserving maps $c:\coprod\limits_{v\in V}\overrightarrow{P}_{1}\rightarrow
G $ with $\omega (t_{v})=v$, $\omega (h_{v})=\pi (v)$, and $\{\omega(h_v) \mid v \in V\}=V$. Collecting
permutomorphic contributors gives:%
\begin{equation*}
\mathrm{perm}(\mathbf{L}_{G})=\sum\limits_{\pi \in S_{V}}\sum\limits_{c\in 
\mathfrak{C}_{\pi }(G)}\prod\limits_{v\in V}\sigma (c(i_{v}))\sigma
(c((j_{v}))\text{.}
\end{equation*}

To calculate the product $\prod\limits_{v\in V}\sigma (c(i_{v}))\sigma
(c((j_{v}))$ first factor out $-1$ for each adjacency determined by $c$,
producing a factor of $(-1)^{oc(c)}$. This forces every negative adjacency
in $G$ appear as a value of $-1$ in $\mathbf{L}_{G}$ and every positive
adjacency in $G$ to appear as a $+1$, since $\mathbf{L}_{G}=\mathbf{D}_{G}-%
\mathbf{A}_{G}$. Now factor out $-1$ from every adjacency that is negative
in $G$, producing a factor of $(-1)^{nc(c)}$ and a net factor of $%
(-1)^{oc(c)+nc(c)}$. Thus,%
\begin{equation*}
\mathrm{perm}(\mathbf{L}_{G})=\sum\limits_{\pi \in S_{V}}\sum\limits_{c\in 
\mathfrak{C}_{\pi }(G)}(-1)^{oc(c)+nc(c)}\text{,}
\end{equation*}%
and combine to get

\begin{equation*}
\mathrm{perm}(\mathbf{L}_{G})=\sum\limits_{c\in \mathfrak{C}%
(G)}(-1)^{oc(c)+nc(c)}\text{.}
\end{equation*}

\textit{Proof of 2. }With the inclusion of the sign of the permutation the
proof is identical until%
\begin{eqnarray*}
\det (\mathbf{L}_{G}) &=&\sum\limits_{\pi \in S_{V}}(-1)^{ec(\pi
)}\sum\limits_{c\in \mathfrak{C}_{\pi }(G)}(-1)^{oc(c)+nc(c)} \\
&=&\sum\limits_{\pi \in S_{V}}\sum\limits_{c\in \mathfrak{C}_{\pi
}(G)}(-1)^{ec(\pi )+oc(c)+nc(c)}\text{.}
\end{eqnarray*}

However, all of the even cycles in $\pi $ correspond to even circles in all
corresponding contributors so%
\begin{eqnarray*}
\det (\mathbf{L}_{G}) &=&\sum\limits_{\pi \in S_{V}}\sum\limits_{c\in 
\mathfrak{C}_{\pi }(G)}(-1)^{ec(c)+oc(c)+nc(c)} \\
&=&\sum\limits_{\pi \in S_{V}}\sum\limits_{c\in \mathfrak{C}_{\pi
}(G)}(-1)^{pc(c)} \\
&=&\sum\limits_{c\in \mathfrak{C}(G)}(-1)^{pc(c)}\text{.}
\end{eqnarray*}

\textit{Proofs of 3 and 4. }The adjacency matrix theorems in parts 3 and 4
are proved similarly, but with the following changes. First, the incidence
preserving maps $\omega $ are replaced with adjacency preserving maps $%
\omega ^{\prime }$ with the same properties. Second, there is no need to
factor out a negative from each adjacency as their signs are accurately
represented in the adjacency matrix. Finally, the sum is over backstep-free
contributors since all the only non-adjacency preserving $\omega $'s are
backsteps. \qed 
\end{proof}

\begin{remark}
As a result from the proof of part 1 of Theorem \ref{LAPD} a backstep is not
considered a circle of a contributor while a loop is considered a circle.
However, from the proof of part 2 any component of a contributor that
corresponds to an algebraic $2$-cycle is considered a circle of a
contributor.
\end{remark}

\subsection{Coefficient Theorems}

Let $\chi ^{D}(\mathbf{M},x):=\det (x\mathbf{I-M)}$ be the determinant-based
characteristic polynomial and $\chi ^{P}(\mathbf{M},x):=\mathrm{perm}(x%
\mathbf{I-M)}$ be the permanent-based characteristic polynomial. Let $bs(c)$
be the number of backsteps in contributor $c$.

\begin{theorem}
\label{MainT}Let $G$ be an oriented hypergraph with adjacency matrix $%
\mathbf{A}_{G}$ and Laplacian matrix $\mathbf{L}_{G}$, then

\begin{enumerate}
\item $\chi ^{P}(\mathbf{A}_{G},x)=\dsum\limits_{k=0}^{\left\vert
V\right\vert }\left( \sum\limits_{c\in \widehat{\mathfrak{C}}%
_{=k}(G)}(-1)^{oc(c)+nc(c)}\right) x^{k}$,

\item $\chi ^{D}(\mathbf{A}_{G},x)=\dsum\limits_{k=0}^{\left\vert
V\right\vert }\left( \sum\limits_{c\in \widehat{\mathfrak{C}}%
_{=k}(G)}(-1)^{pc(c)}\right) x^{k}$,

\item $\chi ^{P}(\mathbf{L}_{G},x)=\dsum\limits_{k=0}^{\left\vert
V\right\vert }\left( \sum\limits_{c\in \widehat{\mathfrak{C}}_{\geq
k}(G)}(-1)^{nc(c)+bs(c)}\right) x^{k}$,

\item $\chi ^{D}(\mathbf{L}_{G},x)=\dsum\limits_{k=0}^{\left\vert
V\right\vert }\left( \sum\limits_{c\in \widehat{\mathfrak{C}}_{\geq
k}(G)}(-1)^{ec(c)+nc(c)+bs(c)}\right) x^{k}$.
\end{enumerate}
\end{theorem}

Before proving the theorem, observe that the presentation of parts 1-4 have
adjacency and Laplacian matrices reversed from Theorem \ref{LAPD}; this is
done as the proofs are parallel based on the appearance of $-\mathbf{A}$.

\begin{proof}

\textit{Proof of 1.} First introduce a choice function $\alpha $ for a given
permutation $\pi $ and vertex $v$ where $\alpha :v\rightarrow \left\{ x\cdot
\delta (v,\pi (v)),\dsum\limits_{\omega ^{\prime }\in \Omega _{1,\pi
}^{\prime }}sgn(\omega ^{\prime }(\overrightarrow{P}_{1}))\right\} $, where $%
\Omega _{1,\pi }^{\prime }$ is the set of all adjacency preserving maps $%
\omega ^{\prime }:\overrightarrow{P}_{1}\rightarrow G$ such that $\omega
^{\prime }(t)=v$ and $\omega ^{\prime }(h)=\pi (v)$. Let $\mathcal{A}_{\pi }$
be the set of all such $\alpha $'s for a given $\pi $.

Observe that if $\pi (v)=v$, then $\alpha $ maps $v$ to either $x$ or to $%
\dsum\limits_{\omega ^{\prime }\in \Omega _{\pi }^{\prime }}sgn(\omega
^{\prime }(\overrightarrow{P}_{1}))$ --- the $(v,v)$-entry of $\mathbf{A}%
_{G} $. However, if $\pi (v)\neq v$, then $\alpha $ maps $v$ to either $0$
or to $\dsum\limits_{\omega ^{\prime }\in \Omega _{\pi }^{\prime
}}sgn(\omega ^{\prime }(\overrightarrow{P}_{1}))$ --- the $(v,\pi (v))$%
-entry of $\mathbf{A}_{G}$. Thus, $\chi ^{P}(\mathbf{A}_{G},x)$ can be
written as%
\begin{eqnarray*}
\chi ^{P}(\mathbf{A}_{G},x) &=&\mathrm{perm}(x\mathbf{I-A}_{G}) \\
&=&\sum\limits_{\pi \in S_{V}}\prod\limits_{v\in V}\sum\limits_{\alpha \in 
\mathcal{A}_{\pi }}\alpha (v)\text{.}
\end{eqnarray*}%
Distributing we get%
\begin{equation*}
=\sum\limits_{\pi \in S_{V}}\sum\limits_{\beta \in \mathcal{B}_{\pi
}}\prod\limits_{v\in V}\beta (v)\text{,}
\end{equation*}%
where $\mathcal{B}_{\pi }$ is the set of all functions $\beta :V\rightarrow
\left\{ x\cdot \delta (v,\pi (v)),\dsum\limits_{\omega ^{\prime }\in \Omega
_{1,\pi }^{\prime }}sgn(\omega ^{\prime }(\overrightarrow{P}_{1}))\right\} $.

Write $\mathcal{B}_{\pi }=\mathcal{B}_{\pi }^{0}\cup \mathcal{B}_{\pi }^{1}$
where $\mathcal{B}_{\pi }^{0}$ is the set of all $\beta $ maps with $\beta
(v)=x\cdot \delta (v,\pi (v))$ for some non-fixed point $v$, and $\mathcal{B}%
_{\pi }^{1}$ is its complement. For every $\beta \in \mathcal{B}_{\pi }^{0}$%
, $\prod\limits_{v\in V}\beta (v)=0$, since there is a non-fixed point $v$
with $\beta (v)=x\cdot \delta (v,\pi (v))=0$. Now partition $\mathcal{B}%
_{\pi }^{1}$ into $\bigcup\limits_{k=0}^{\left\vert V\right\vert }\mathcal{B}%
_{k,\pi }$, where $\mathcal{B}_{k,\pi }$ is the set of all $\beta \in 
\mathcal{B}_{\pi }^{1}$ such that $\left\vert \beta ^{-1}(x)\right\vert =k$.
For each $\beta \in \mathcal{B}_{k,\pi }$ let $U_{\beta }\subseteq V$ be the
set of $\left\vert V\right\vert -k$ vertices not mapped to $x$ by $\beta $,
giving%
\begin{equation*}
=\sum\limits_{\pi \in S_{V}}\sum\limits_{k=0}^{\left\vert V\right\vert
}\sum\limits_{\beta \in \mathcal{B}_{k,\pi }}\left[ \left(
\prod\limits_{v\in U_{\beta }}\beta (v)\right) x^{k}\right] \text{.}
\end{equation*}

Let $\Omega _{1,\pi }^{\prime }[U_{\beta }]$ be the set of adjacency
preserving maps $\omega ^{\prime }:$ $\overrightarrow{P}_{1}\rightarrow
(G\Nwseline \overline{U}_{\beta })$ such that $\omega ^{\prime }(t)=v$, 
$\omega ^{\prime }(h)=\pi (v)$, and $G\Nwseline \overline{U}_{\beta }$
is the oriented-hypergraph resulting from the weak-deletion of vertices of $%
\overline{U}_{\beta }$ and the deletion of any resulting $0$-edges. Evaluating $\beta (v)$ gives%
\begin{equation*}
=\sum\limits_{\pi \in S_{V}}\sum\limits_{k=0}^{\left\vert V\right\vert
}\sum\limits_{\beta \in \mathcal{B}_{k,\pi }}\left[ \left(
\prod\limits_{v\in U_{\beta }}\dsum\limits_{\omega ^{\prime }\in \Omega
_{1,\pi }^{\prime }[U_{\beta }]}sgn(\omega ^{\prime }(\overrightarrow{P}%
_{1}))\right) x^{k}\right] \text{,}
\end{equation*}%
and distributing again produces%
\begin{equation*}
=\sum\limits_{\pi \in S_{V}}\sum\limits_{k=0}^{\left\vert V\right\vert
}\sum\limits_{\beta \in \mathcal{B}_{k,\pi }}\left[ \left(
\dsum\limits_{c\in \widehat{\mathfrak{C}}_{=0,\pi }(G\Nwseline 
\overline{U}_{\beta })}\prod\limits_{v\in U_{\beta }}\sigma (c(i_{v}))\sigma
(c((j_{v}))\right) x^{k}\right] \text{,}
\end{equation*}%
where $\widehat{\mathfrak{C}}_{=0,\pi }(G\Nwseline \overline{U}_{\beta
})$ is the set of all backstep-free permutomorphic contributors $%
c:\coprod\limits_{v\in U_{\beta }}\overrightarrow{P}_{1}\rightarrow
(G\Nwseline \overline{U}_{\beta })$ with $c(t_{v})=v$, $c(h_v)=\pi(v)$, and 
$\{c(h_{v})\mid v\in U_{\beta }\}=U_{\beta }$.

Note the $\omega ^{\prime }\in \Omega _{1,\pi }^{\prime }[U_{\beta }]$ are
adjacency preserving so any $1$-edges resulting from weak deletion of $%
\overline{U}_{\beta }$ cannot be mapped onto as backsteps. As a result,
weak-deletion is not needed, and the sub-hypergraph induced on vertex set $%
U_{\beta }$ would be sufficient. However, the preservation of all the
incidences containing $U_{\beta }$ is necessary for incidence preserving
maps when determining Laplacian coefficients. Thus, $G\Nwseline 
\overline{U}_{\beta }$ is the smallest sub-object in which all results are
true.

Factoring out $-1$ as in the proof for part 1 of Theorem \ref{LAPD} gives%
\begin{equation*}
=\sum\limits_{\pi \in S_{V}}\sum\limits_{k=0}^{\left\vert V\right\vert
}\sum\limits_{\beta \in \mathcal{B}_{k,\pi }}\left[ \left(
\dsum\limits_{c\in \widehat{\mathfrak{C}}_{=0,\pi }(G\Nwseline 
\overline{U}_{\beta })}(-1)^{oc(c)+nc(c)}\right) x^{k}\right] \text{.}
\end{equation*}

Again, the $\beta \in \mathcal{B}_{k,\pi }$ are determined by $\overline{U}%
_{\beta }\subseteq V$ and $c\in \widehat{\mathfrak{C}}_{=0,\pi
}(G\Nwseline \overline{U}_{\beta })$ are backstep-free contributors on $%
U_{\beta }$. Observe that for each $c\in \widehat{\mathfrak{C}}_{=0,\pi
}(G\Nwseline \overline{U}_{\beta })$ there is a natural extension to $G$
via elements of $\widehat{\mathfrak{C}}_{=k,\pi }(G)$ by including an
isolated vertex for each vertex in $\overline{U}_{\beta }$. Specifically,
the extension factors

\begin{equation*}
\begin{tikzpicture}[node distance = 1.25in, auto] \node
(TL){$\coprod\limits_{v\in U_{\beta
}}\overrightarrow{P}_{1}\cup\overline{U}_{\beta }$}; \node (TR) [right
of=TL,xshift=.5in] {$G$}; \node (BL)[below of=TL] {$\coprod\limits_{v\in
U_{\beta }}\overrightarrow{P}_{1}$}; \node (BR) [below of=TR]
{$G\Nwseline \overline{U}_{\beta }$}; \draw[->] (TL) to node
{$c\in\widehat{\mathfrak{C}}_{=k,\pi}(G)$} (TR); \draw[right hook->] (BL) to
node [] {} (TL); \draw[->] (BL) to node []
{$c\in\widehat{\mathfrak{C}}_{=0,\pi }(G\Nwseline \overline{U}_{\beta
})$} (BR); \draw[right hook->] (BR) to node [swap]{} (TR); \end{tikzpicture}
\end{equation*}%
and combines the sum as%
\begin{equation*}
=\sum\limits_{\pi \in S_{V}}\sum\limits_{k=0}^{\left\vert V\right\vert } 
\left[ \left( \dsum\limits_{c\in \widehat{\mathfrak{C}}_{=k,\pi
}(G)}(-1)^{oc(c)+nc(c)}\right) x^{k}\right] \text{,}
\end{equation*}%
where the elements of $\widehat{\mathfrak{C}}_{=k,\pi }(G)$ are
backstep-free with $k$ isolated vertices (which determine $\overline{U}%
_{\beta }$) and adjacency preserving on the non-isolated vertices.

Reverse the order of the first two summations and combining permutomorphic
contributors as in Theorem \ref{LAPD} yields%
\begin{equation*}
=\sum\limits_{k=0}^{\left\vert V\right\vert }\left( \dsum\limits_{c\in 
\widehat{\mathfrak{C}}_{=k}(G)}(-1)^{oc(c)+nc(c)}\right) x^{k}\text{.}
\end{equation*}%
Completing the proof for part 1.

\textit{Proof of 2. }Mirrors the proof of part 2 in Theorem \ref{LAPD}. The
inclusion of the sign of the permutation the proof is identical until

\begin{eqnarray*}
\chi ^{D}(\mathbf{A}_{G},x) &=&\sum\limits_{k=0}^{\left\vert V\right\vert
}\sum\limits_{\pi \in S_{V}}(-1)^{ec(\pi )}\left[ \left( \dsum\limits_{c\in 
\widehat{\mathfrak{C}}_{=k,\pi }(G)}(-1)^{oc(c)+nc(c)}\right) x^{k}\right] \\
&=&\sum\limits_{k=0}^{\left\vert V\right\vert }\sum\limits_{\pi \in S_{V}} 
\left[ \left( \dsum\limits_{c\in \widehat{\mathfrak{C}}_{=k,\pi
}(G)}(-1)^{ec(\pi )+oc(c)+nc(c)}\right) x^{k}\right] \\
&=&\sum\limits_{k=0}^{\left\vert V\right\vert }\sum\limits_{\pi \in S_{V}} 
\left[ \left( \dsum\limits_{c\in \widehat{\mathfrak{C}}_{=k,\pi
}(G)}(-1)^{ec(c)+oc(c)+nc(c)}\right) x^{k}\right] \\
&=&\sum\limits_{k=0}^{\left\vert V\right\vert }\left( \dsum\limits_{c\in 
\widehat{\mathfrak{C}}_{=k}(G)}(-1)^{pc(c)}\right) x^{k}\text{.}
\end{eqnarray*}

\textit{Proofs of 3 and 4.} Both $\chi ^{D}(\mathbf{L}_{G},x)$ and $\chi
^{P}(\mathbf{L}_{G},x)$ are proved similarly, but with the following
changes. First, replace adjacency preserving maps with incidence preserving
maps. Second, $x\mathbf{I-L}_{G}$ has the sign of adjacencies represented
correctly while the degrees are negated, so factor out a $-1$ for each
backstep to produce $(-1)^{bs(c)}$, and then factor out a $-1$ for each
negative adjacency producing $(-1)^{nc(c)}$. Finally, distributing incidence
preserving maps produces an inner sum over $\widehat{\mathfrak{C}}_{\geq
0,\pi }(G\Nwseline \overline{U}_{\beta })$. \qed
\end{proof}

\begin{remark}
Sachs' Theorem (Theorem \ref{Sachs}) is a corollary of Theorem \ref{MainT}
if $G$ is a graph and oriented circles are combined.
\end{remark}

\begin{remark}
Comparing $\mathrm{perm}(-\mathbf{L}_{G})=\chi ^{P}(\mathbf{L}_{G},0)$ we
see 
\begin{equation*}
\dsum\limits_{c\in \mathfrak{C}_{\geq 0}(G)}(-1)^{oc(c)+nc(c)+\left\vert
V\right\vert }=\sum\limits_{c\in \mathfrak{C}_{\geq 0}(G)}(-1)^{nc(c)+bs(c)}%
\text{.}
\end{equation*}%
This suggests the parity of $bs(c)$ is equal to the parity of $%
oc(c)+\left\vert V\right\vert $. This can easily be verified directly by
cases, and helps translate between Theorems \ref{LAPD} and \ref{MainT}.
\end{remark}

\subsubsection{Examples}

\paragraph{$\mathbf{A}_{G}$\textbf{\ for an oriented }$3$\textbf{-circuit}}

Consider the incidence-oriented $3$-circuit in Figure \ref{C3Perm} below
with its contributors grouped by permutomorphism classes.

\begin{figure}[H]
\centering
\includegraphics{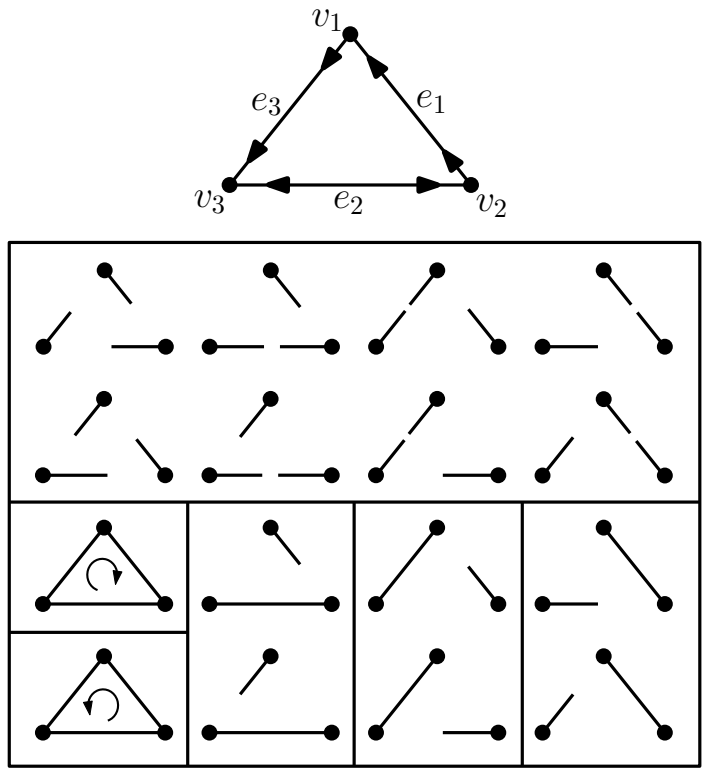}
\caption{An incidence-oriented $3$-circuit and its contributors grouped by
permutomorphism classes.}
\label{C3Perm}
\end{figure}

The oriented $3$-circuit has adjacency matrix%
\begin{equation*}
\mathbf{A}_{G}=\left[ 
\begin{array}{ccc}
0 & 1 & 1 \\ 
1 & 0 & -1 \\ 
1 & -1 & 0%
\end{array}%
\right]
\end{equation*}%
with $\chi ^{D}(\mathbf{A}_{G},x)=x^{3}-3x+2$. From part 2 of Theorem \ref%
{MainT} the sub-contributor signing function is $(-1)^{pc(c)}$. Observe that
there are only two backstep-free contributors, each with $0$ positive
circles, so the constant is $(-1)^{0}+(-1)^{0}=2$.

The coefficient of $x^{1}$ is the signed sum of elements of $\widehat{%
\mathfrak{C}}_{=1}(G)$. The elements of $\widehat{\mathfrak{C}}_{=1}(G)$ for
the incidence-oriented $3$-circuit are: 
\begin{figure}[H]
\centering
\includegraphics{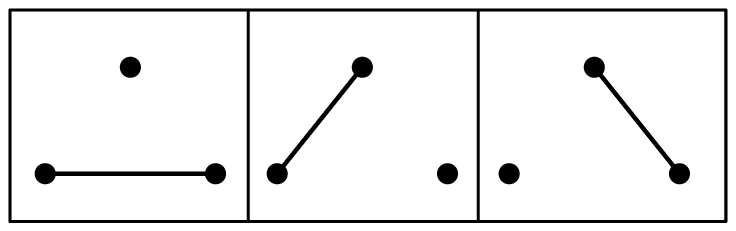}
\caption{Elements of $\protect\widehat{\mathfrak{C}}_{=1}(G)$ for an
incidence-oriented $3$-circuit grouped by permutomorphism classes.}
\label{C3PermEq1}
\end{figure}
From Figure \ref{C3PermEq1} we see there are three elements in $\widehat{%
\mathfrak{C}}_{=1}(G)$, each with a single degenerate $2$-circle (which are
necessarily positive) so the coefficient of $x^{1}$ is $%
(-1)^{1}+(-1)^{1}+(-1)^{1}=-3$.

The remaining coefficients are obtained by observing that $\widehat{%
\mathfrak{C}}_{=(\left\vert V\right\vert -1)}(G)$ is necessarily empty, and $%
\widehat{\mathfrak{C}}_{=\left\vert V\right\vert }(G)$ only contains the set
of isolated vertices.

\paragraph{$\mathbf{L}_{G}$\textbf{\ for an oriented }$3$\textbf{-circuit}}

The oriented $3$-circuit in Figure \ref{C3Perm} has Laplacian matrix 
\begin{equation*}
\mathbf{L}_{G}=\left[ 
\begin{array}{ccc}
2 & -1 & -1 \\ 
-1 & 2 & 1 \\ 
-1 & 1 & 2%
\end{array}%
\right]
\end{equation*}%
with $\chi ^{D}(\mathbf{L}_{G},x)=x^{3}-6x^{2}+9x-4$. From part 4 of Theorem %
\ref{MainT} the sub-contributor signing function is $%
(-1)^{ec(c)+nc(c)+bs(c)} $. The contributors in the first column in Figure %
\ref{C3Perm} are all $-1$ while the remaining three columns sum each sum to $%
0$, producing a constant $-4$.

The coefficient of $x^{1}$ is the signed sum of elements of $\widehat{%
\mathfrak{C}}_{\geq 1}(G)$. The elements of $\widehat{\mathfrak{C}}_{\geq
1}(G)$ for the incidence-oriented $3$-circuit are: 
\begin{figure}[H]
\centering
\includegraphics{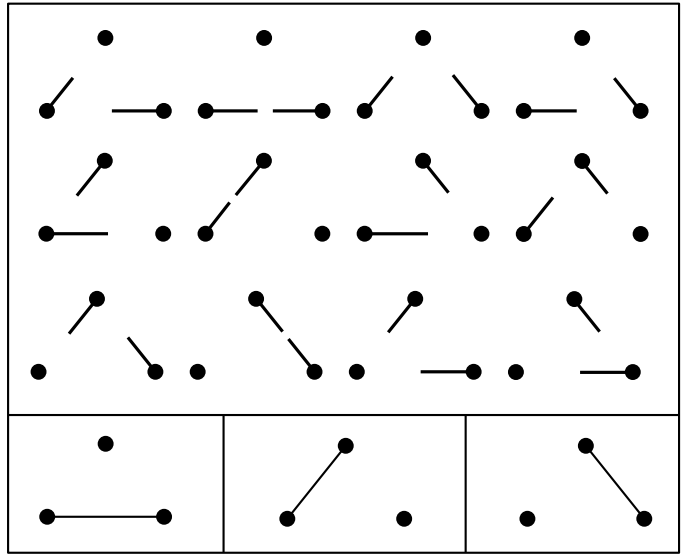}
\caption{Elements of $\protect\widehat{\mathfrak{C}}_{\geq 1}(G)$ for an
incidence-oriented $3$-circuit grouped by permutomorphism classes.}
\label{C3PermGEq1}
\end{figure}
From Figure \ref{C3PermGEq1} there are fifteen elements in $\widehat{%
\mathfrak{C}}_{\geq 1}(G)$, the twelve arising from the identity permutation
are $+1$ while the three with degenerate $2$-circles are $-1$, producing a
coefficient of $12-3=+9$ for $x^{1}$.

The remaining coefficients can be checked similarly.

\paragraph{\textbf{An extroverted }$3$\textbf{-edge}}

Consider the extroverted-oriented $3$-edge in and its contributors in Figure %
\ref{E3Perm}. 
\begin{figure}[H]
\centering
\includegraphics{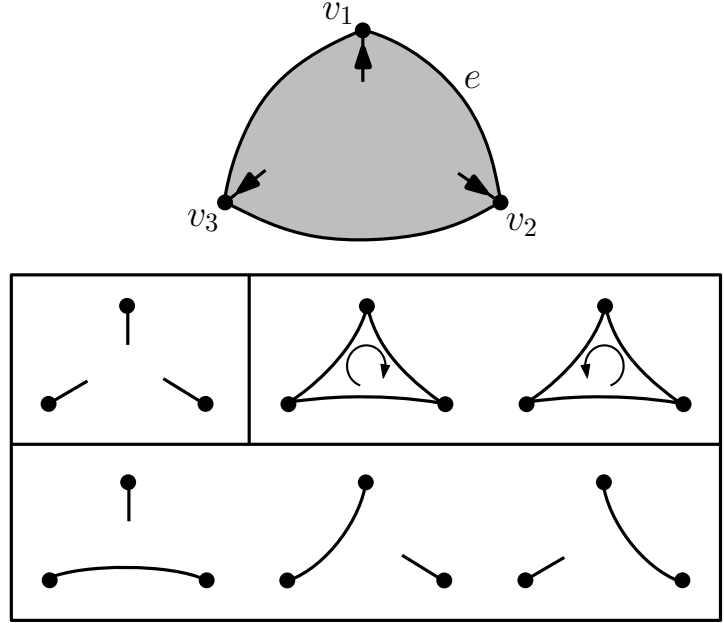}
\caption{An incidence-oriented $3$-edge and its contributors grouped by
permutomorphism classes.}
\label{E3Perm}
\end{figure}
The extroverted $3$-edge has adjacency matrix%
\begin{equation*}
\mathbf{A}_{G}=\left[ 
\begin{array}{ccc}
0 & -1 & -1 \\ 
-1 & 0 & -1 \\ 
-1 & -1 & 0%
\end{array}%
\right]
\end{equation*}%
with $\chi ^{D}(\mathbf{A}_{G},x)=x^{3}-3x+2$, and Laplacian matrix 
\begin{equation*}
\mathbf{L}_{G}=\left[ 
\begin{array}{ccc}
1 & 1 & 1 \\ 
1 & 1 & 1 \\ 
1 & 1 & 1%
\end{array}%
\right]
\end{equation*}%
with $\chi ^{D}(\mathbf{L}_{G},x)=x^{3}-3x^{2}$.

The constant of $\chi ^{D}(\mathbf{A}_{G},x)$ is $2$ as there are only two
backstep-free contributors, each with $0$ positive circles. The constant of $%
\chi ^{D}(\mathbf{L}_{G},x)$ is $0$ as the contributors in the top row
contributors cancel those in the bottom row. The remaining coefficients can
be checked by referring to Figure \ref{E3Perm}.

\subsection{Optimizing Permanents and Determinants}

Theorems \ref{LAPD} and \ref{MainT} allow for the study of optimizing
permanents and determinants through contributors. The bounds presented are
for integral incidence matrices --- it may be possible to sharpen of the
bounds via complex unit orientations by combining with the work from \cite%
{OHHar} and \cite{Reff1}.

\begin{theorem}
\label{PLMax}For a fixed underlying oriented hypergraph $G$ with no isolated
vertices, no $0$-edges, and varied orientation function $\sigma $, the
following are equivalent:

\begin{enumerate}
\item $\mathrm{perm}(\mathbf{L}_{G})$ is maximized over all orientations $%
\sigma $,

\item $\sigma $ is the all extroverted or all introverted orientation,

\item $\mathbf{L}_{G}$ is the signless Laplacian,

\item $\mathrm{perm}(\mathbf{L}_{G})=\left\vert \mathfrak{C}(G)\right\vert $.
\end{enumerate}
\end{theorem}

\begin{proof}
Part $2$ and $3$ are trivially equivalent. 
To see the equivalence of parts $1$ and $3$ observe 
that in the signless Laplacian a circle in an associated contributor is negative if, and only if, it is odd. 
From Theorem \ref{LAPD}, $\perm(\mathbf{L}_{G})$ is maximal when, for every contributor $c$, $oc(c)$ and $nc(c)$ 
have the same parity. However, if $oc(c) \neq nc(c)$, then there either exists a circle in a contributor which is either 
odd and not negative, or negative and not odd. Refine the corresponding algebraic cycle into fixed elements 
causing each element in the circle to become a backstep, forming a new contributor $c'$. The parity 
of $oc(c')$ and $nc(c')$ are not equal, so $oc(c) = nc(c)$ for all $c\in \mathfrak{C}(G)$.

Moreover, when $oc(c) = nc(c)$ for all contributors $c$, then a circle of $c$ is odd if, and only if, it is negative by 
a similar refinement argument as above. Thus, $\perm(\mathbf{L}_{G})$ is maximal if, and only if, $L_G$ is the signless Laplacian.

The equivalence for part $4$ is obvious by Theorem \ref{LAPD}. \qed
\end{proof}

\begin{theorem}
\label{Bounds}Let $G$ be an oriented hypergraph with no isolated vertices or 
$0$-edges with Laplacian matrix $\mathbf{L}_{G}$, then

\begin{enumerate}
\item $-\left\vert \mathfrak{C}(G)\right\vert <\mathrm{perm}(\mathbf{L}%
_{G})\leq \left\vert \mathfrak{C}(G)\right\vert $, and $\mathrm{perm}(%
\mathbf{L}_{G})=\left\vert \mathfrak{C}(G)\right\vert $ if, and only if, $G$
is extroverted or introverted,

\item $-\left\vert \mathfrak{C}(G)\right\vert <\det (\mathbf{L}_{G})\leq
\left\vert \mathfrak{C}(G)\right\vert $, and $\det (\mathbf{L}%
_{G})=\left\vert \mathfrak{C}(G)\right\vert $ if, and only if, the connected
components of $G$ consist of bouquets of introverted or extroverted $k$%
-edges.
\end{enumerate}
\end{theorem}

\begin{proof}

From \ref{LAPD} it is clear that each of $\mathrm{perm}(\mathbf{L}_{G})$ and 
$\det (\mathbf{L}_{G})$ are in the interval $[-\left\vert \mathfrak{C}%
(G)\right\vert ,\left\vert \mathfrak{C}(G)\right\vert ]$.

\textit{Proof of 1. }From \ref{LAPD} $\mathrm{perm}(\mathbf{L}%
_{G})=\left\vert \mathfrak{C}(G)\right\vert $ if, and only if, for every
contributor $c$, $oc(c)$ and $nc(c)$ have the same parity. Also, $%
\mathrm{perm}(\mathbf{L}_{G})=-\left\vert \mathfrak{C}(G)\right\vert $ if,
and only if, for every contributor $c$, $oc(c)$ and $nc(c)$ have different
parity. From Theorem \ref{PLMax} $\mathrm{perm}(\mathbf{L}_{G})=\left\vert 
\mathfrak{C}(G)\right\vert $ if, and only if, $G$ is extroverted or
introverted. To see $\mathrm{perm}(\mathbf{L}_{G})\neq -\left\vert \mathfrak{%
C}(G)\right\vert $ use a similar refinement argument as in \ref{PLMax} on
the unequal number of $oc(c)$ and $nc(c)$.

\textit{Proof of 2.} From \ref{LAPD} $\det (\mathbf{L}_{G})=\left\vert 
\mathfrak{C}(G)\right\vert $ if, and only if, every contributor has an even
number of positive circles, while $\det (\mathbf{L}_{G})=-\left\vert 
\mathfrak{C}(G)\right\vert $ if, and only if, every contributor has an odd
number of positive circles. To see that $\det (\mathbf{L}_{G})\neq
-\left\vert \mathfrak{C}(G)\right\vert $ observe that there is at least one
contributor with $0$ positive circles, since there is at least one
contributor corresponding to the identity permutation.

If the connected components of $G$ consist of a bouquet of introverted or
extroverted $k$-edges, there are $0$ positive circles, and $\det (\mathbf{L}%
_{G})=\left\vert \mathfrak{C}(G)\right\vert $. To see the converse, observe
that if $G$ contains a non-loop adjacency, then there is a contributor with
the resulting degenerate $2$-circle as the only circle, and it is
necessarily positive. Therefore, the only adjacencies $G$ can contain are
negative loops, and the result follows. \qed 
\end{proof}

\begin{example}
The oriented $3$-edge in Figure \ref{E3Perm} is extroverted with $6$
contributors so $\mathrm{perm}(\mathbf{L}_{G})=6$.
\end{example}

A refinement argument for adjacency matrices cannot be used on elements of $%
\mathfrak{C}_{=0}(G)$ as they are necessarily backstep-free. Moreover, for
oriented hypergraphs, if a circle of a contributor lies within a single edge
of $G$, then the odd/even parity of the circle determines if it is
negative/positive. This presents some complications as these
contributor-circles are not circles in the oriented hypergraph. It is easy
to see that this prevents balanced oriented hypergraphs from being a class
of oriented hypergraphs where $\mathrm{perm}(\mathbf{A}_{G})=\left\vert 
\mathfrak{C}_{=0}(G)\right\vert $ --- consider an oriented hypergraph
consisting of a $4$-edge and a $2$-edge that share a single vertex, and the
contributor consisting of the degenerate $2$-circle for the $2$-edge and the 
$3$-circle on the remaining $4$-edge vertices. However, the following signed
graphic theorem is true:

\begin{theorem}
\label{Bounds2}If $G$ is a balanced signed graph, then $\mathrm{perm}(%
\mathbf{A}_{G})$ is maximal and equals $\left\vert \mathfrak{C}%
_{=0}(G)\right\vert $.
\end{theorem}

\begin{proof}

From Theorem \ref{LAPD} $\mathrm{perm}(\mathbf{A}_{G})=\left\vert \mathfrak{C%
}_{=0}(G)\right\vert $ if, and only if, every contributor has an even number
of negative circles. A balanced signed graph has $0$ negative circles, while
degenerate $2$-circles are positive, so every element of $\mathfrak{C}%
_{=0}(G) $ has no negative circles. \qed 
\end{proof}

Balanced signed graphs are not the only class of signed graphs where $%
\mathrm{perm}(\mathbf{A}_{G})=\left\vert \mathfrak{C}_{=0}(G)\right\vert $,
as the negative $3$-circuit in Figure \ref{C3Perm} is maximal, it is easy to
check that negative circuit graphs are also optimal.

The determinant of $\mathbf{A}_{G}$ presents similar issues. However,
bouquets of $1$-edges and positive $2$-edge loops, and positive circuits of
length $\equiv 0\mod4$, can easily be checked to optimize $\det (\mathbf{A}%
_{G})$.

\begin{problem}
Characterize the oriented hypergraphs that optimize the $x^{k}$ coefficient
of a given characteristic polynomial.
\end{problem}

\section*{Acknowledgements}

This research is partially supported by Texas State University Mathworks.
The authors sincerely thank the referee for careful reading the manuscript
and for their valuable feedback.


\section*{References}

\bibliographystyle{amsplain2}
\bibliography{mybib}

\end{document}